\documentclass[10pt,draft,reqno]{amsart}
     \makeatletter
     \def\section{\@startsection{section}{1}%
     \z@{.7\linespacing\@plus\linespacing}{.5\linespacing}%
     {\bfseries
     \centering }}
     \def\@secnumfont{\bfseries}
     \makeatother
\newtheorem{theorem}{Theorem}
\newtheorem{lemma}{Lemma}

\newtheorem{corollary}{Corollary}
\theoremstyle{definition}

\theoremstyle{remark}
\newtheorem*{remark}{Remark}
\numberwithin{equation}{section}
\usepackage{eucal,hyperref}

\begin{document}

UDK: 519.218.2 \\

\title[Slowly varying functions in theory of Markov Branching Processes]
    {On application of slowly varying functions \\ with remainder in the theory of
        Markov \\ Branching Processes  with mean one and \\ infinite variance}

\author{{Azam Imomov and Abror Meyliyev}}
\address{Azam Abdurakhimovich Imomov and Abror Khujanazarovich Meyliyev
\newline\hphantom{iii} Karshi State University,
\newline\hphantom{iii} 17, Kuchabag street,
\newline\hphantom{iii} 180100 Karshi city, Uzbekistan.}
\email{{{imomov{\_}\,azam@mail.ru}}}
\thanks{\copyright \ 2019 Imomov A.A and Meyliyev A.Kh.}



\dedicatory{Dedicated to our Teachers}

\vspace{.4cm}

\begin{abstract}
    We investigate an application of slowly varying functions (in sense of Karamata) in
    the theory of Markov branching process. We treat the critical case so that the
    infinitesimal generating function of the process has the infinite second  moment,
    but it regularly varies with remainder. We improve the Basic Lemma of the theory
    of critical Markov branching process and refine well-known limit results.
\end{abstract}

\maketitle

\section{Introduction and main results}   \label{MySec:1}

    \subsection{Preliminaries}       \label{MySubsec:1.1}

    We consider the Markov Branching Process (MBP) to be the homogeneous continuous-time Markov process
    $\left\{Z(t), t \geq {0} \right\}$ with the  state space ${\mathcal S}_0 = \{0\} \cup {\mathcal S}$,
    where ${\mathcal S} \subset \mathbb{N}$ and $\mathbb{N}=\{1,2, \ldots \}$. The transition probabilities
    of the process
\[
    P_{ij}(t) := \mathbb{P}\left\{{{Z(t) = j} \bigl| {Z(0)=i}\bigr.}\right\}
\]
    satisfy the following branching property:
\begin{equation}           \label{1.1}
    P_{ij} (t) = P_{1j}^{i*} (t)
    \quad \parbox{2.8cm}{\textit{for all} {} $i,j \in {\mathcal S}$,}
\end{equation}
    where the asterisk denotes convolution. Herein transition
    probabilities $P_{1j}(t)$ are expressed by relation
\begin{equation}           \label{1.2}
    P_{1j} (\varepsilon ) = \delta _{1j}  + a_j \varepsilon  + o(\varepsilon )
    \quad \parbox{2cm}{\textit{as} {} $\varepsilon  \downarrow 0$,}
\end{equation}
    where $\delta _{ij}$ is Kronecker's delta function and $\left\{ {a_j } \right\}$ are intensities of
    individuals' transformation such that $a_j  \ge 0$ for $j \in {\mathcal S}_0 \backslash \{ 1\} $ and
\begin{equation*}
    0 < a_0  <  - a_1  = \sum\limits_{j \in {\mathcal S}_0 \backslash \{1\}}{a_j} < \infty.
\end{equation*}
    The MBP was defined first by Kolmogorov and Dmitriev {\cite{KolmDm}}; for more
    detailed information see {\cite[Ch.~III]{ANey}} and {\cite[Ch.~V]{Harris63}}.

    Defining the generating function (GF) $F(t;s) = \sum\nolimits_{j \in {\mathcal S}_0}{P_{1j}(t)s^j}$
    it follows from \eqref{1.1} and \eqref{1.2} that the process $\left\{Z(t)\right\}$ is determined
    by the infinitesimal GF $f(s)=\sum\nolimits_{j \in {\mathcal S}_0}{a_j s^j}$ for $s\in [0, 1)$.
    Moreover it follows from \eqref{1.2} that GF $F(t;s)$ is unique solution of the backward
    Kolmogorov equation ${\partial{F} \mathord{\left/{\vphantom{{\partial{F}}{\partial{t}}}}
    \right.\kern-\nulldelimiterspace}{\partial{t}}}=f\left({F}\right)$ with the boundary
    condition $F(0;s) = s$; see {\cite[p.~106]{ANey}}. If $m: = \sum\nolimits_{j \in {\mathcal S}}
    {ja_j }= f'(1-)$ is finite then $F(t;1)= 1$ and due to Kolmogorov equation it can be calculated
    that $\mathbb{E}\bigl[{{Z(t)}\left| {Z(0)=i}\right.}\bigr]= \sum\nolimits_{j \in{\mathcal S}}
    {jP_{ij}(t)} = ie^{mt}$. Last formula shows that long-term properties of MBP are various
    depending on value of parameter $m$. Hence the MBP is classified as \textit{critical}
    if $m = 0$ and \textit{sub-critical} or \textit{supercritical} if $m < 0$ or $m > 0$
    respectively. Monographs {\cite{AsHer}}--{\cite{Bingham}} and {\cite{Harris63}}
    are general references for mentioned and other classical facts on theory of MBP.

    In the paper we consider the critical case. Let $R(t;s)=1-F(t;s)$ and
\begin{equation*}
    q(t):=R(t;0) = \mathbb{P}\left\{{\mathcal{H}>t}\right\},
\end{equation*}
    where the variable $\mathcal{H} = \inf\left\{t: Z(t)=0\right\}$ denotes an extinction
    time of MBP. Then $q(t)$ is the survival probability of the process. Sevastyanov \cite{Sev51}
    proved that if $f{'''}(1-)<\infty$ then the following asymptotic representation holds:
\begin{equation}           \label{1.3}
    {{1} \over {R(t;s)}}-{{1}\over {1-s}}={{f{''}(1-)}\over {2}}{t}+{\mathcal{O}\bigl(\ln{t}\bigr)}
    \quad \parbox{2.2cm}{\textit{as} {} $t  \rightarrow \infty$}
\end{equation}
    for all $s\in[0,1)$; see {\cite[p.~72]{Sev51}}.

    Later on Zolotarev {\cite{Zol57}} has found a principally new result on asymptotic
    representation of $q(t)$ without the assumption of $f{''}(1-)<\infty$. Namely
    providing that $g(x)=f(1-x)$ is a regularly varying function at zero that is
\begin{equation*}
    \lim_{x\downarrow{0}} {{xg'(x)} \over {g(x)}} = \gamma
\end{equation*}
    with index $1<\gamma = 1+\alpha \leq{2}$, he has proved that
\begin{equation}           \label{1.4}
    {{q(t)} \over {f\left({1-q(t)}\right)}} \sim  {\alpha}t
    \quad \parbox{2cm}{\textit{as} {} $t \rightarrow {\infty}$.}
\end{equation}

    Further we assume that the infinitesimal GF $f(s)$ has the following representation:
\begin{equation}           \label{1.5}
    f(s) = (1 - s)^{1 + \nu} \mathcal{L}\left( {{{1} \over {1 - s}}}\right)
\end{equation}
    for all $s\in [0, 1)$, where $0 < \nu < 1$ and $\mathcal{L}(x)$ is slowly varying (SV)
    function at infinity (in sense of Karamata; see {\cite{SenetaRV}}).

    Pakes {\cite{Pakes2010}}, in connection with the proof of limit theorems has
    established, that if the condition \eqref{1.5} holds then
\begin{equation}           \label{1.6}
    {{1} \over {R\left(t;s\right)}} = {U\left(t+V\left({1} \over {1-s}\right)\right)},
\end{equation}
    where $V(x)=\mathcal{M}\left( 1-{1/x}\right)$ and $\mathcal{M}(s)$ is GF of invariant
    measure of MBP that is $\mathcal{M}(s)=\sum\nolimits_{j \in {\mathcal S}}{\mu_j {s^j}}$
    and $\sum\nolimits_{i \in {\mathcal S}}{\mu_i {P_{ij}(t)}}=\mu_j$, $j\in {\mathcal S}$.
    Function $U(y)$ is the inverse of $V(x)$. The formula \eqref{1.6} gives
    an alternative relation to \eqref{1.4}:
\begin{equation*}
    q(t) = {{1} \over {U\left(t\right)}}.
\end{equation*}

    The following lemma is a version of more recent result that was proved in
    {\cite[second part statement of Lemma~1]{Imomov17}}, in which the character of asymptotical
    decreasing of the function $R(t;s)$ seems to be more explicit rather than in \eqref{1.6}.

\begin{lemma}               \label{MyLem:1}
    If the condition \eqref{1.5} holds then
\begin{equation}           \label{1.7}
    R(t;s) = {{{\mathcal N}(t)} \over {(\nu t)^{{1 \mathord{\left/
    {\vphantom {1 \nu }} \right. \kern-\nulldelimiterspace} \nu }} }}
    \cdot \left[ {1 - {{M(t;s)} \over {\nu{t}}}} \right],
\end{equation}
    where
\begin{equation}           \label{1.8}
    {\mathcal N}^{\,\nu}(t) \cdot \mathcal{L}\left({{{\bigl(\nu t \bigr)^{{1
    \mathord{\left/ {\vphantom {1 \nu }} \right. \kern-\nulldelimiterspace}
    \nu }} } \over  {{\mathcal N}(t)}}} \right) \longrightarrow 1
    \quad \parbox{2.2cm}{\textit{as} {} $t  \rightarrow \infty$.}
\end{equation}
    Herein $M(t;0)=0$ for all $t>0$ and $M(t;s)\rightarrow {\mathcal{M}(s)}$ as
    $t\rightarrow {\infty}$, where $\mathcal{M}(s)$ is GF of invariant measures of MBP and
\begin{equation*}
    \mathcal{M}(s) = \int_1^{{1 \mathord{\left/ {\vphantom {1 {(1 - s)}}}
    \right. \kern-\nulldelimiterspace}{(1 - s)}}}
    {{{dx} \over {x^{1 - \nu }\mathcal{L}(x)}}}.
\end{equation*}
\end{lemma}

\subsection{Aim and Basic assumptions}   \label{MySubsec:1.2}

    The representation \eqref{1.5} implies that the second moment $2b:=f{''}(1-)=\infty$.
    If $b < \infty$ then it takes place with $\nu=1$ and $\mathcal{L}(t)\rightarrow{b}$
    as $t \to \infty$ and we can write asymptotic formula in type of \eqref{1.3}. This
    circumstance suggests that we can look for some sufficient condition such that an
    asymptotic relation similar to \eqref{1.3} will be true provided that \eqref{1.5}
    holds. So the aim of the paper is to improve the Lemma~\ref{MyLem:1} and
    thereafter to refine \eqref{1.4} and to improve some earlier well-known
    results by imposing an additional condition on the function $\mathcal{L}(s)$.

    Let
\begin{equation*}
    \Lambda (y):= y^\nu  \mathcal{L}\left( {{1 \over {y}}} \right)
\end{equation*}
    for $y\in (0, 1]$ and rewrite \eqref{1.5} as
\begin{equation*}
    f(1-y) = y\Lambda{(y)}.   \leqno[\textsf {$f_\nu$}]
\end{equation*}
    Note that the function $y\Lambda (y)$ is positive, tends to zero and has a monotone
    derivative so that ${{y\Lambda '(y)} \mathord{\left/ {\vphantom {{y\Lambda '(y)}
    {\Lambda (y)}}}\right. \kern-\nulldelimiterspace} {\Lambda (y)}} \to \nu$ as
    $y \downarrow 0$; see {\cite[p.~401]{Bingham}}. Thence it is natural to write
\begin{equation*}
    {{y\Lambda '(y)} \over {\Lambda (y)}} = \nu + \delta(y),    \leqno[\textsf {$\Lambda_{\delta}$}]
\end{equation*}
    where $\delta(y)$ is continuous and $\delta(y) \to 0$ as $y \downarrow 0$.

    Throughout the paper $\left[{f_{\nu}} \right]$ and
    $\left[{\Lambda_{\delta}} \right]$ are our \textit{Basic assumptions}.

    Since ${{{\mathcal{L}\left( {\lambda x} \right)}} \mathord{\left/ {\vphantom {{{\mathcal{L}
    \left( {\lambda x}\right)}}{{\mathcal{L}\left(x \right)}}}}\right.\kern-\nulldelimiterspace}
    {{\mathcal{L}\left(x\right)}}} \to {1}$ as $x \to \infty$ for each $\lambda >0$ we can write
\begin{equation}          \label{1.9}
    {{\mathcal{L}\left( {\lambda x} \right)} \over {\mathcal{L}(x)}} = 1 + \varrho(x),
\end{equation}
    where $\varrho (x) \to 0$ as $x \to \infty$. If there is some positive function
    $g(x)$ so that $g(x) \to 0$ and $\varrho (x)={\mathcal{O}} \bigl(g(x)\bigr)$ as $x \to \infty $,
    then $\mathcal{L}(x)$ is said to be \textit{SV-function with remainder} at infinity;
    see {\cite[p.~185, condition SR1]{Bingham}}. As we can see below, if the function
    $\delta(y)$ is known it will be possible to estimate a decreasing rate of the remainder $\varrho (x)$.

    Using that $\Lambda (1)=\mathcal{L}(1)=a_0$ integration $\left[{\Lambda_{\delta}} \right]$ yields
\begin{equation*}
    \Lambda(y)=a_{0}y^{\nu} \exp{\int\limits_1^y {{{\delta(u)}\over{u}}\textrm{d}u}}.
\end{equation*}
    Therefore we have
\begin{equation*}
    \mathcal{L}\left( {{1 \over y}} \right)
    = a_{0} \exp {\int\limits_1^y {{{\delta(u)} \over {u}}\textrm{d}u}}.
\end{equation*}
    Changing variable as $u = {{1} \mathord{\left/ {\vphantom {{1} t}}
    \right. \kern-\nulldelimiterspace} t}$ in the integrand gives
\begin{equation}          \label{1.10}
    \mathcal{L}\left( x \right)
    = a_{0} \exp {\int\limits_1^x {{{\varepsilon (t)} \over {t}}\textrm{d}t}},
\end{equation}
    where $\varepsilon(t)=-\delta(1/t)$ and $\varepsilon(t) \to 0$ as $t \to \infty$.
    It follows from \eqref{1.9} and \eqref{1.10} that
\begin{equation*}
    {{\mathcal{L}(\lambda x)} \over {\mathcal{L}(x)}}
    = \exp \int\limits_{x}^{\lambda x} {{{\varepsilon (t)} \over t}\textrm{d}t}
    = 1 + \varrho (x) \quad \parbox{2.2cm}{\textit{as} {} $x  \to \infty$}
\end{equation*}
    for each $\lambda > 0$, where $\varrho(x) \to 0$ as $x \to \infty$. Thus
\begin{equation*}
    \int\limits_{x}^{\lambda x} {{{\varepsilon (t)} \over t}\textrm{d}t} =
    \ln \left[{1 + \varrho (x)}\right] = \varrho (x)+{\mathcal O}
    \left( {\varrho ^2 (x)} \right)
    \quad \parbox{2.2cm}{\textit{as} {} $x  \to \infty$.}
\end{equation*}
    Applying the mean value theorem to the left-hand side of the
    last equality we can assert that
\begin{equation}          \label{1.11}
    \varepsilon(x)={\mathcal O}\left({\varrho(x)}\right)
    \quad \parbox{2.2cm}{\textit{as} {} $x  \to \infty$.}
\end{equation}

    Thus the assumption $\left[{\Lambda_{\delta}}\right]$ provides that $\mathcal{L}(s)$
    to be an SV-function at infinity with the remainder in the form of
    $\varrho (x) = {\mathcal O}\left( {\delta \bigl(1/x\bigr)} \right)$ as $x  \to \infty$.

\subsection{Results}   \label{MySubsec:1.3}

    Our results appear due to an improvement of the Lemma~\ref{MyLem:1}
    under the Basic assumptions. Let
\begin{equation*}
    {\mho(t;s)} := {\int\limits_0^t {{\delta\left(R(u;s)\right)}\textrm{d}u}},
    \quad {\mho(t)} :={\mho(t;0)}.
\end{equation*}
    Needless to say $R(t;s)\to {0}$ as $t \to \infty$, due to \eqref{1.7}.
    Therefore, since $\delta(y) \to 0$ as $y \downarrow 0$ we make sure of
\begin{equation*}
    {{\mho{(t;s)}} \over {t}} = {{1} \over {t}} {\int\limits_0^t {{\delta\left(R(u;s)\right)}\textrm{d}u}}=o(1)
    \quad \parbox{2.2cm}{\textit{as} {} $t \to \infty$.}
\end{equation*}
    Thus ${\mho{(t;s)}}=o(t)$ as $t \to \infty$. Herewith a more important interest represents the special case when
\begin{equation}         \label{1.12}
    \delta (y) = \Lambda (y).
\end{equation}

\begin{remark}
     The case \eqref{1.12} implies that $\mathcal{L}\left( x \right)$
     be an SV-function at infinity with the remainder in the form of
\begin{equation*}
    \varrho(x) = {\mathcal O}\left({{\mathcal{L}(x)} \over {x^{\nu}}}\right)
    \quad \parbox{2.2cm}{\textit{as} {} $x  \to \infty$.}
\end{equation*}
     So under the condition \eqref{1.12} our results appear for all
     SV-functions at infinity with remainder $\varrho(x)$ in the form above.
\end{remark}

\begin{theorem}                 \label{MyTh:1}
     Under the Basic assumptions
\begin{equation}         \label{1.13}
    {q(t)} = {{{\mathcal{N}}(t)} \over {\bigl(\nu t \bigr)^{{1 \mathord{\left/
    {\vphantom {1 \nu}}\right. \kern-\nulldelimiterspace} \nu}}}}\left({1-{{\mho{\left(t\right)}}
    \over {\nu^{2}t}} + {o}\left({{\mho{(t)}} \over {t}}\right)} \right)
    \quad \parbox{2.2cm}{\textit{as} {} $t  \to \infty$,}
\end{equation}
    herein and everywhere ${\mathcal N}(t)$ is the SV-function
    satisfying \eqref{1.8}. In addition, if \eqref{1.12} holds then
\begin{equation}         \label{1.14}
    {q(t)} = {{{\mathcal{N}}(t)} \over {\bigl(\nu t \bigr)^{{1 \mathord{\left/
    {\vphantom {1 \nu}}\right. \kern-\nulldelimiterspace} \nu}}}}
    \left({1-{{\ln{\left[a_{0}\nu{t}+1\right]}} \over {\nu^{3}t}}
    + {o}\left({{\ln{t}} \over {t}}\right)} \right)
    \quad \parbox{2.2cm}{\textit{as} {} $t  \to \infty$.}
\end{equation}
\end{theorem}

\begin{theorem}                 \label{MyTh:2}
    Under the Basic assumptions
\begin{equation}         \label{1.15}
    {\bigl(\nu t \bigr)^{1 + {1 \mathord{\left/ {\vphantom {1 \nu }} \right.
    \kern-\nulldelimiterspace} \nu }}} \cdot P_{11}(t) = {{\mathcal{N}(t)}
    \over  {a_0}}\left({1-{{{1+\nu}}\over{\nu^{2}}}{{\mho{\left(t\right)}}
    \over{t}} + {o}\left({{\mho{(t)}} \over {t}}\right)} \right)
\end{equation}
    as $t \to \infty$. In addition, if \eqref{1.12} holds then
\begin{equation}         \label{1.16}
    {\bigl(\nu t \bigr)^{1 + {1 \mathord{\left/ {\vphantom {1 \nu }} \right.
    \kern-\nulldelimiterspace} \nu }}} \cdot P_{11}(t) = {{\mathcal{N}(t)}
    \over  {a_0}}\left( {1 - {{1 + \nu} \over  {\nu^3}}{{\ln{\left[a_{0}\nu{t}+1\right]}} \over t}
    + {o}\left({{\ln{t}} \over {t}}\right)} \right)
\end{equation}
    as $t \to \infty$.
\end{theorem}

    Let $\mathbb{P}_{i}\bigl\{{*}\bigr\}:= \mathbb{P}\left\{{{\ast} \bigl| {Z(0)=i}\bigr.}\right\}$
    and consider a conditional distribution
\begin{equation*}
    \mathbb{P}_i^{{\mathcal H}(t+u)}\{*\}: = \mathbb{P}_i
    \left\{{*\bigl| {t+u<{\mathcal H}< \infty} \bigr.} \right\}.
\end{equation*}
    It was shown in {\cite{Imomov12}} that the probability measure
\begin{equation}          \label{1.17}
    {\mathcal Q}_{ij} (t): = \mathop {\lim }\limits_{u \to \infty }
    \mathbb{P}_i^{{\mathcal{H}}(t+u)}\{Z(t)=j\}={{j}\over{i}}P_{ij}(t)
\end{equation}
    defines the continuous-time Markov chain $\left\{ {W(t)}, t\geq{0} \right\}$ with states
    space ${\mathcal E} \subset \mathbb{N}$, called the \textit{Markov Q-process} (MQP).
    According to the definition
\begin{equation*}
    {\mathcal Q}_{ij}(t) = \mathbb{P}_i \left\{{Z(t)=j
    \bigl| {{\mathcal {H}} = \infty } \bigr.} \right\},
\end{equation*}
    so MQP can be interpreted as MBP with non degenerating trajectory in remote future.

    In a term of GF the equality \eqref{1.17} can be written as following:
\begin{equation}          \label{1.18}
    G_{i}(t;s): = \sum\limits_{j \in {\mathcal E}} {{\mathcal Q}_{ij}(t)s^j }
    = \bigl[{F(t;s)}\bigr]^{i-1}G(t;s),
\end{equation}
    where GF $G(t;s):=G_{1}(t;s)=\mathbb{E}\left[{s^{W(t)}
    \bigl| {W(0) = 1} \bigr.} \right]$ and
\begin{equation*}
    G(t;s)=-s{{\partial{R(t;s)} \over {\partial{s}}}}
    \quad \parbox{2.2cm}{\textit{for all} {} ${t}\geq{0}$.}
\end{equation*}
    Combining the backward and the forward Kolmogorov equations
    we write it in the next form
\begin{equation}          \label{1.19}
    G(t;s)=s{{f\left(F(t;s)\right) } \over {f(s)}}
    \quad \parbox{2.2cm}{\textit{for all} {} ${t}\geq{0}$.}
\end{equation}

    Since $F(t;s)\to{1}$ as $t\to{\infty}$ uniformly for all $s\in[0, 1)$ according to \eqref{1.18}
    it is suffice to consider the case $i=1$.

\begin{theorem}         \label{MyTh:3}
    Under the Basic assumptions
\begin{equation}         \label{1.20}
    {(\nu t)^{1+{1\mathord{\left/ {\vphantom {1 \nu}}\right.\kern-\nulldelimiterspace}\nu}}}
    G(t;s) = \pi(s){{\mathcal{N}}(t)} \bigl({1+{\rho}\left(t;s\right)}\bigr),
\end{equation}
    where the function $\pi(s)$ has an expansion in powers of $s$ with non-negative coefficients
    so that $\pi (s) = \sum\nolimits_{j \in {\mathcal E}}{\pi_j s^j}$ and $\left\{{\pi_j},
    j\in{\mathcal E}\right\}$ is an invariant measure for MQP. Moreover it has a form of
\begin{equation}         \label{1.21}
    \pi (s) = {s \over {{(1-s)}^{1+\nu}}} \mathcal{L}_{\pi}\left({1 \over {{1-s}}}\right),
\end{equation}
    where $\mathcal{L}_{\pi}(\ast)=\mathcal{L}^{-1}(\ast)$. Furthermore
    ${\rho}\left(t;s\right)=o(1)$ as $t\to{\infty}$. In addition, if \eqref{1.12} holds then
\begin{equation}           \label{1.22}
    {\rho}\left(t;s\right)= -{{{1+\nu}\over{\nu^3}}{{\ln{\left[{\Lambda{\left(1-s\right)}}\nu{t}+1\right]}}
    \over t} + {o}\left({{\ln{t}} \over {t}}\right)}
    \quad \parbox{2.2cm}{\textit{as} {} $ t \to \infty $.}
\end{equation}
\end{theorem}

    Note that in accordance with Tauberian theorem for the power series
    (see {\cite[Ch.~XIII, \S~5, p.~513, Theorem~5]{Feller}}) the relation \eqref{1.21} implies
\begin{equation*}
    \sum\limits_{j=1}^n{\pi_j} ~ \sim {1 \over {\Gamma(2+\nu)}}n^{1+\nu} \mathcal{L}_{\pi}(n)
    \quad \parbox{2.2cm}{\textit{as} {} $ n \to \infty $,}
\end{equation*}
    where $\Gamma (*)$ is Euler's Gamma function and
    $\left(\mathcal{L}_{\pi}\cdot \mathcal{L}\right)(\ast)= 1$.

    Let $D(t; x): = \mathbb{P}\bigl\{{{q(t)W(t) \le x}}\bigr\}$.
    In {\cite[Theorem~21]{Imomov17}} it was proved, that if $\left[{f_{\nu}} \right]$ holds then
\begin{equation*}
     \lim_{t\to{\infty}}{D(t; x)} ={D(x)},
\end{equation*}
    where
\begin{equation*}
     \Psi (\theta) := {\int\nolimits_{0}^{\infty}}{e^{-\theta x}\textrm{d}D(x)}
     = {{1} \over {\left({1+\theta^{\nu}}\right)^{1+{1 \mathord{\left/
    {\vphantom {1 \nu }}\right. \kern-\nulldelimiterspace}\nu}}}}.
\end{equation*}

\begin{theorem}         \label{MyTh:4}
    Let
\begin{equation*}
     \Delta(t; \theta):= \left|{\int\nolimits_{0}^{\infty}
     {e^{ - \theta x} \textrm{d}D(t;x)} - \Psi (\theta)}\right|.
\end{equation*}
    If the Basic assumptions and \eqref{1.12} hold then
\begin{equation}           \label{1.23}
    \mathop{\sup}\limits_{\theta \in {\left(0, \infty \right)}} \Delta(t; \theta)
    = {{1+\nu }\over {\nu^3}}{{\ln{t}}\over t}\bigl({1+o(1)}\bigr)
    \quad \parbox{2.2cm}{\textit{as} {} $ t \to \infty $.}
\end{equation}
\end{theorem}

    Theorem~\ref{MyTh:4} yields that from Berry-Esseen type inequality
    (see {\cite[Ch.~XVI, \S~3, p.~616, Lemma~2]{Feller}}) follows

\begin{corollary}         \label{MyCor}
    Under the conditions of Theorem~\ref{MyTh:4}
\begin{equation}           \label{1.24}
    \mathop{\sup}\limits_{x \in {\left(0, \infty \right)}} \Bigl|{{D(t; x)}-{D(x)}}\Bigr|
    = \mathcal{O}\left({{\ln{t}}\over t}\right)
    \quad \parbox{2.2cm}{\textit{as} {} $ t \to \infty $.}
\end{equation}
\end{corollary}

\section{Auxiliaries}   \label{MySec:2}

    The following lemma improves the statement of the Lemma~\ref{MyLem:1}.

\begin{lemma}                 \label{MyLem:2}
    Under the Basic assumptions
\begin{equation}               \label{2.1}
    {{1}\over{\Lambda \left(R(t;s)\right)}}-{{1}\over {\Lambda \left(1-s\right)}}
    = \nu{t} + {\int\limits_0^t {{\delta\left(R(u;s)\right)}\textrm{d}u}}.
\end{equation}
    If in addition \eqref{1.12} holds then
\begin{equation}          \label{2.2}
    {{1} \over {\Lambda \left(R(t;s)\right)}}-{{1}\over {\Lambda \left(1-s\right)}}
    =\nu{t}+{{\,1}\over{\,\nu}}\,{\ln{\nu(t;s)}}+{o}{\bigl(\ln{\nu(t;s)}\bigr)}
\end{equation}
    as $t \to \infty $, where $\nu(t;s)={\Lambda{\left(1-s\right)}}\nu{t}+1$.
\end{lemma}

\begin{proof}
    From $\left[{\Lambda_{\delta}} \right]$ we write
\begin{equation}          \label{2.3}
    {{R\Lambda{'}\left(R\right)} \over {\Lambda \left(R\right)}} = \nu + \delta\left(R\right)
\end{equation}
    since $R=R(t;s)\to{0}$ as $t\to{\infty}$. By the backward Kolmogorov equation ${\partial{F} \mathord{\left/{\vphantom
    {{\partial{F}}{\partial{t}}}}\right.\kern-\nulldelimiterspace}{\partial{t}}}=f\left({F}\right)$
    and considering representation $\left[{f_{\nu}}\right]$ the relation \eqref{2.3} becomes
\begin{equation*}
    {{\textrm{d}\Lambda\left(R\right)} \over {\textrm{d}t}}=-{{\Lambda\left(R\right)}\over {R}}
    f\left(1-R\right)\bigl(\nu + \delta\left(R\right)\bigr)=-{{\Lambda^{2}\left(R\right)}}
    \bigl(\nu + \delta\left(R\right)\bigr).
\end{equation*}
    Therefore
\begin{equation}          \label{2.4}
    \textrm{d}\left[{{1} \over {\Lambda\left(R\right)}} - \nu{t}\right]= \delta\left(R\right)\textrm{d}t.
\end{equation}
    Integrating \eqref{2.4} from $0$ to $t$ we obtain \eqref{2.1}.

    To prove \eqref{2.2} we should calculate integral in \eqref{2.1}
    putting $\delta (y) = \Lambda (y)$. Write
\begin{equation}          \label{2.5}
    {{1}\over{\Lambda \left(R(t;s)\right)}}-{{1}\over {\Lambda \left(1-s\right)}}
    = \nu{t} + {\int\limits_0^t {{\Lambda\left(R(u;s)\right)}\textrm{d}u}}.
\end{equation}
    Since $\Lambda{(y)}=y^{\nu} \mathcal{L}\left({{1 \mathord{\left/ {\vphantom {1 {y} }} \right.
    \kern-\nulldelimiterspace} y}}\right)$ and $R(t;s)\rightarrow 0$ as $t\rightarrow{\infty}$
    for $s\in [0, 1)$, the integral in the right-hand side of \eqref{2.5} is ${o}(t)$. Hence
\begin{equation*}
    {\Lambda \left(R(t;s)\right)}={{\Lambda{\left(1-s\right)}}\over {{\nu(t;s)}}}
    +{o}{\left({{\Lambda{\left(1-s\right)}}\over {{\nu(t;s)}}}\right)}
    \quad \parbox{2.2cm}{\textit{as} {} $t  \to \infty$,}
\end{equation*}
    where $\nu(t;s)={\Lambda{\left(1-s\right)}}\nu{t}+1$. Therefore
\begin{equation}          \label{2.6}
    \mho(t;s)={\int\limits_0^t {{\Lambda\left(R(u;s)\right)}\textrm{d}u}}=
    {{\,1\,}\over{\nu}}\,{\ln{\nu(t;s)}}+{o}{\bigl(\ln{\nu(t;s)}\bigr)}
    \quad \parbox{2.2cm}{\textit{as} {} $t  \to \infty$.}
\end{equation}
    This together with \eqref{2.5} implies \eqref{2.2}.
\end{proof}

    In the proof of our results we also will essentially use the following lemma.

\begin{lemma}                 \label{MyLem:3}
    Let
\begin{equation*}
    \phi{(y)} := y - yK\left(y\right),
\end{equation*}
    where $K\left(y\right)\to{0}$ as $y\downarrow{0}$.
    If in addition to the Basic assumptions \eqref{1.12} holds, then
\begin{equation}          \label{2.7}
    \mathcal{L}\left( {{1 \over {\phi{(y)}}}} \right)
    = \mathcal{L}\left({{1 \over {y}}} \right)
    \left({1+{\mathcal{O}}\bigl({\Lambda(y)}\bigr)}\right)
    \quad \parbox{2.2cm}{\textit{as} {} $y  \downarrow 0$.}
\end{equation}
\end{lemma}

\begin{proof}
    Since the function ${\mathcal {L}}(x) = x^\nu \Lambda \left( {{1 \mathord{\left/
    {\vphantom {1 x}} \right. \kern-\nulldelimiterspace} x}} \right)$ is
    differentiable, by virtue of the mean value theorem we have
\begin{equation}          \label{2.8}
    \mathcal{L}\left( {{x \over {1 - K}}} \right) - \mathcal{L}(x)
    = \mathcal{L}'\left( {{{1 - \gamma{K} } \over {1 - K}}x}
    \right) \cdot {{K}\over{1-K}}x  ,
\end{equation}
    where $K:=K\left({{1\mathord{\left/{\vphantom{1 x}}\right.\kern-\nulldelimiterspace}x}}
    \right)$ and $0<\gamma <1$. Since $\varrho(x)={\mathcal O}{\left(\mathcal{L}(x)/x^{\nu}\right)}$,
    from \eqref{1.10} and \eqref{1.11} it follows that
\begin{equation}          \label{2.9}
    \mathcal{L}'(u) = \mathcal{L}(u){ {{\varepsilon (u)} \over u}}
    = {\mathcal{O}}\left({{{\mathcal{L}^{2}(u)}\over {u^{1+\nu}}}}\right)
    \quad \parbox{2.2cm}{\textit{as} {} $u  \rightarrow \infty$.}
\end{equation}
    Denote $u = {{\left({1 -\gamma {K}}\right)x}\mathord{\left/
    {\vphantom {{\left({1- \gamma {K}} \right)x}{\left({1-K}
    \right)}}} \right. \kern-\nulldelimiterspace}{\left({1-K}\right)}}$.
    Since  $K\left({{1\mathord{\left/{\vphantom {1 x}}\right.
    \kern-\nulldelimiterspace} x}} \right) \to 0$ then $u\sim x$ and
    $\mathcal{L}(u)\sim \mathcal{L} (x)$ as $x \to \infty$. Therefore
    after using \eqref{2.9} in the equality \eqref{2.8} and some
    elementary transformations the assertion \eqref{2.7} readily follows.
\end{proof}

\section{The proofs of results}

\begin{proof}[Proof of Theorem~\ref{MyTh:1}]
    Putting $s=0$ in \eqref{2.1} we have
\begin{equation}          \label{3.1}
    {{1}\over{\Lambda \left(q(t)\right)}} = \nu{t}+{{1}\over {a_0}}+ {\mho{(t)}}
\end{equation}
    and by elementary arguments we get to assertion \eqref{1.13}.
    Similarly putting $s=0$ in \eqref{2.2}, we obtain \eqref{1.14}.
\end{proof}

\begin{proof}[Proof of Theorem~\ref{MyTh:2}]
    Considering together the backward and the forward Kolmogorov equations
    and seeing $\left[{f_{\nu}} \right]$ we write
\begin{equation*}
    {{\partial F(t;s)}\over{\partial s}}={{f\left({1-R(t;s)}\right)}\over{f(s)}}
    = {{R(t;s)\Lambda\left(R(t;s)\right) } \over {f(s)}}.
\end{equation*}
    Thence at $s=0$ we deduce
\begin{equation*}
    {P_{11}(t)} = {{q(t)\Lambda\left(q(t)\right) } \over {a_0}}.
\end{equation*}
    Hence using \eqref{1.13} and \eqref{1.14} the relations \eqref{1.15} and \eqref{1.16} easily follow.
\end{proof}

\begin{proof}[Proof of Theorem~\ref{MyTh:3}]
    It follows from \eqref{1.19} and $\left[{f_{\nu}} \right]$ that
\begin{equation}          \label{3.2}
    {G(t;s)}={{R^{1+\nu}(t;s)}\over{f(s)}}\mathcal{L}\left({{1\over{R(t;s)}}}\right).
\end{equation}
    On the other hand \eqref{2.1} entails
\begin{equation}          \label{3.3}
    R(t;s) = {{\mathcal{N}(t;s)} \over {(\nu t)^{{1 \mathord{\left/{\vphantom {1 \nu }}
    \right. \kern-\nulldelimiterspace}\nu}}}}{\left({1-{{{{\mho(t;s)}}}\over{{\nu}^{2}t}}
    {\left(1+{o}(1)\right)}} \right)}
    \quad \parbox{2.2cm}{\textit{as} {} $t  \to \infty$.}
\end{equation}
    where $\mathcal{N}(t;s):= {\mathcal{L}^{-{1 \mathord{\left/{\vphantom {1 \nu}} \right.
    \kern-\nulldelimiterspace} \nu }} \left( {{1 \mathord{\left/{\vphantom {1{R(t;s)}}}
    \right.\kern-\nulldelimiterspace}{R(t;s)}}}\right)}$ and ${\mho(t;s)} =
    {\int\nolimits_0^t {{\delta\left(R(u;s)\right)}\textrm{d}u}}=o(t)$
    as $t \rightarrow{\infty}$. From \eqref{3.3} we conclude that
\begin{equation*}
    R(t;s) = {q(t)} {{\mathcal{N}(t;s)} \over {\mathcal{N}(t)}}
    {\left({1-{{{{\mho(t;s)}}}\over{{\nu}^{2}t}}{\left(1+{o}(1)\right)}}\right)}
    \quad \parbox{2.2cm}{\textit{as} {} $t  \to \infty$.}
\end{equation*}
    Since ${R(t;s)}/{q(t)}\to{1}$ uniformly for $s\in[0, 1)$ then
    ${\mathcal{N}(t;s)}/{\mathcal{N}(t)}\to{1}$ for all $s\in[0, 1)$.
    But in accordance with \eqref{1.9} and \eqref{1.12}
\begin{equation*}
    {{\mathcal{L}\left({R^{-1}(t;s)}\right)} \over
    {\mathcal{L}\left(q^{-1}(t)\right)}}=1+\mathcal{O}{\left({1}\over{t}\right)}
\end{equation*}
    and therefore
\begin{equation}          \label{3.4}
    {{\mathcal{N}(t;s)}\over{\mathcal{N}(t)}}=1+\mathcal{O}{\left({1}\over{t}\right)}
    \quad \parbox{2.2cm}{\textit{as} {} $t  \to \infty$.}
\end{equation}
    Combining $\left[{f_{\nu}} \right]$ and \eqref{3.2}--\eqref{3.4} we obtain
\begin{equation}          \label{3.5}
    {G(t;s)} = {{\pi{(s)}\mathcal{N}(t)}\over {(\nu t)^{1+{1 \mathord{\left/{\vphantom {1 \nu}}\right.
    \kern-\nulldelimiterspace} \nu}}}}{\left({1-{{{{1+\nu}}}\over{{\nu}^{2}}}{{{{\mho(t;s)}}}\over{t}}
    {\left(1+{o}(1)\right)}} \right)}
    \quad \parbox{2.2cm}{\textit{as} {} $t  \to \infty$.}
\end{equation}
    The representation \eqref{1.20} with evanescent \eqref{1.22} follows from
    \eqref{2.6} and \eqref{3.5}.

    To show that $\pi{(s)}$ is GF of invariant measure,
    from \eqref{1.19} we obtain the following functional equation:
\begin{equation*}
    G(t + \tau ;s) = {{G\left(t;s\right)}
    \over {F(t;s)}} {G\left({\tau; F(t;s)} \right)}
    \quad \parbox{3cm}{\textit{for all} {} $\tau >0$}
\end{equation*}
    since $F(t + \tau ;s) = {F\left({\tau; F(t;s)} \right)}$; see {\cite{Sev51}}.
    Then taking limit as $\tau \to {\infty}$ it follows from this equation that
\begin{equation*}
    \pi(s) = {{G\left(t;s\right)} \over {F(t;s)}} {\pi\left({F(t;s)} \right)}.
\end{equation*}
    This is equivalent to the equation
\begin{equation*}
    \pi _j = \sum\limits_{i \in {\mathcal E}}{\pi _i {\mathcal Q}_{ij} (t)}.
\end{equation*}
    Thus $\left\{{\pi_j}, j\in{\mathcal E}\right\}$ is an invariant measure for MQP.
\end{proof}

\begin{proof}[Proof of Theorem~\ref{MyTh:4}]
        Consider the Laplace transform
\begin{equation*}
    \Psi(t;\theta):=\mathbb{E}e^{-\theta q(t)W(t)}= G\left( {t;\theta (t)} \right),
\end{equation*}
    where $\theta(t)=\exp \{-\theta q(t)\}$.
    From $\left[{f_{\nu}} \right]$ and \eqref{1.19} we write
\begin{equation}          \label{3.6}
    \Psi(t;\theta)=\theta(t)\cdot\left({{R\left({t;\theta(t)}\right)}\over{1-\theta(t)}}\right)^{1+\nu}
    \cdot {{{\mathcal{L} \left( {{1 \mathord{\left/{\vphantom {1{R\left(t;\theta(t)\right)}}}
    \right.\kern-\nulldelimiterspace}{R\left(t;\theta(t)\right)}}}\right)}} \over
    {{\mathcal{L} \left( {{1 \mathord{\left/{\vphantom {1{\left(1-\theta(t)\right)}}}
    \right.\kern-\nulldelimiterspace}{\left(1-\theta(t)\right)}}}\right)}}}.
\end{equation}

    It follows from \eqref{2.2} that
\begin{equation}          \label{3.7}
    {{1} \over {\Lambda \left(R\left(t;\theta(t)\right)\right)}}-{{1}\over
    {\Lambda \left(1-\theta(t)\right)}}=\nu{t}+{{\,1}\over{\,\nu}}\,{\ln{\left[{\Lambda
    {\left(1-\theta(t)\right)}}\nu{t}+1\right]}}+{o}{\bigl(\ln{t}\bigr)}
\end{equation}
    as $t \to \infty $. Since $1-e^{-x}\sim {x-{x^{2}/2}}$ as $x\to{0}$ then according to our designation
\begin{equation*}
    {\Lambda \left(1-\theta(t)\right)}={\theta^{\nu}}q^{\nu}(t)
    \mathcal{L}\left({{1}\over {1-\theta(t)}}\right) {\left(1-{{1}\over{2}}
    {\theta}q(t) \left(1+{o}{\bigl(1\bigr)}\right)\right)^{\nu}}
\end{equation*}
    as $t \to \infty$. By Lemma~\ref{MyLem:3} with $K(y)={y}/{2}$
\begin{equation}          \label{3.8}
    \mathcal{L}\left({{1}\over {1-\theta(t)}}\right)= \mathcal{L}\left({{1}\over {q(t)}}\right)
    \left(1+\mathcal{O}{\left({1}\over{t}\right)}\right)
    \quad \parbox{2.2cm}{\textit{as} {} $t  \to \infty$.}
\end{equation}
    Then
\begin{equation*}
    {\Lambda \left(1-\theta(t)\right)}={\theta^{\nu}}\Lambda\left({q(t)}\right)
    \left(1+\mathcal{O}{\left({1}\over{t}\right)}\right)
    \quad \parbox{2.2cm}{\textit{as} {} $t  \to \infty$,}
\end{equation*}
    since $q(t)=\mathcal{O}{\left({\mathcal{N}(t)}/{t^{1/\nu}}\right)}$ and $\nu <1$.
    Thence considering \eqref{3.1}
\begin{equation}          \label{3.9}
    {\Lambda \left(1-\theta(t)\right)}={{{\theta^{\nu}}}\over{\nu{t}}}
    \left(1- {{1}\over{\nu^{2}}}{{\ln{t}}\over{t}}\left(1+{o}{\bigl(1\bigr)}\right)\right)
    \quad \parbox{2.2cm}{\textit{as} {} $t  \to \infty$.}
\end{equation}

    Using \eqref{3.9} we can write \eqref{3.7} in the following form:
\begin{equation*}
    {{1} \over {\Lambda \left(R\left(t;\theta(t)\right)\right)}}={\nu{t}}{{{1+\theta^{\nu}}}\over{\theta^\nu}}
    \left(1- {{1}\over{1+\theta^\nu}}{{\ln{t}}\over{\nu^{2}t}}\left(1+{o}{\bigl(1\bigr)}\right)\right)
\end{equation*}
    and, therefore
\begin{equation}          \label{3.10}
    R\left(t;\theta(t)\right) = {{{\mathcal{N}_\theta}(t)} \over
    {(\nu{t})^{{1 \mathord{\left/ {\vphantom {1 \nu }} \right.\kern-
    \nulldelimiterspace} \nu }}}}{\theta  \over {\bigl( {1 + \theta ^\nu }
    \bigr)^{{1 \mathord{\left/{\vphantom {1 \nu }} \right.
    \kern-\nulldelimiterspace} \nu }}}}\left({1-{1 \over
    {{1 + \theta ^\nu}}}{{\ln{t}} \over{\nu^{3}t}}\bigl({1+o(1)}\bigr)}\right)
\end{equation}
    as $t \to \infty$, where $\mathcal{N}_\theta(t):= {\mathcal{L}^{-{1 \mathord{\left/{\vphantom{1 \nu}}
    \right.\kern-\nulldelimiterspace}\nu}}\left({{1\mathord{\left/{\vphantom{1{R\left(t;\theta(t)\right)}}}
    \right.\kern-\nulldelimiterspace}{R\left(t;\theta(t)\right)}}}\right)}$.

    Since ${{R(t;s)} \mathord{\left/ {\vphantom {{R(t;s)}{q(t)}}}
    \right. \kern-\nulldelimiterspace}{q(t)}}\to 1$ for all $s \in[0, 1)$,
    then by force of \eqref{3.10} it is necessary that
\begin{equation*}
    {{R\left(t;\theta(t)\right)} \over{q(t)}} \longrightarrow c(\theta)
    \quad \parbox{2.2cm}{\textit{as} {} $t  \to \infty$,}
\end{equation*}
    where $|c(\theta)|<\infty$ at any fixed $\theta \in {\left(0, \infty \right)}$.
    Therefore according to \eqref{1.9}
\begin{equation}          \label{3.11}
    {{{\mathcal L} \left({R^{-1}\left(t;\theta(t)\right)}\right)} \over
    {{\mathcal L} \left( {q^{-1}(t)} \right)}} = 1 +
    \mathcal{O}\bigl(\Lambda \left({q(t)} \right)\bigr)
    \quad \parbox{2.2cm}{\textit{as} {} $t  \to \infty$}
\end{equation}
    or the same that
\begin{equation*}
    {{{\mathcal{N}}_{\theta}(t)} \over {{\mathcal{N}}(t)}}
    = 1 + \mathcal{O}\left( {{1} \over {t}} \right)
    \quad \parbox{2.2cm}{\textit{as} {} $t  \to \infty$.}
\end{equation*}
    Thus \eqref{3.10} becomes
\begin{equation}          \label{3.12}
    R\left(t;\theta(t)\right) = {{{\mathcal{N}}(t)} \over
    {(\nu{t})^{{1 \mathord{\left/ {\vphantom {1 \nu }} \right.\kern-
    \nulldelimiterspace} \nu }}}}{\theta  \over {\bigl( {1 + \theta ^\nu }
    \bigr)^{{1 \mathord{\left/{\vphantom {1 \nu }} \right.
    \kern-\nulldelimiterspace} \nu }}}}\left({1-{1 \over
    {{1 + \theta ^\nu}}}{{\ln{t}} \over{\nu^{3}t}}\bigl({1+o(1)}\bigr)}\right)
\end{equation}
    as $t \to \infty$.

    Further using \eqref{3.8} and \eqref{3.11} we can rewrite \eqref{3.6} as
\begin{equation*}
    \Psi(t;\theta)=\left({{R\left({t;\theta(t)}\right)}\over{1-\theta(t)}}\right)^{1+\nu}
    \left(1+\mathcal{O}{\left({1}\over{t}\right)}\right)
    \quad \parbox{2.2cm}{\textit{as} {} $t  \to \infty$}
\end{equation*}
    and using \eqref{3.12} after some transformation we obtain
\begin{equation}          \label{3.13}
    \Psi(t;\theta)=\Psi(\theta)
    \left({1+{{\theta^\nu}\over{{1+\theta^\nu}}}{{1+\nu}\over{{\nu^3}}}
    {{\ln{t}}\over{t}}\bigl({1+o(1)}\bigr)}\right)
    \quad \parbox{2.2cm}{\textit{as} {} $t  \to \infty$.}
\end{equation}
    The assertion \eqref{1.23} follows from \eqref{3.13}.
\end{proof}

\par\medskip\noindent
{\bf Acknowledgment.} The authors are grateful to the anonymous referee for careful reading of the
    manuscript and for his kindly comments which contributed to improving the paper.

\par\bigskip\noindent

\bibliographystyle{amsplain}

\end{document}